\documentclass[11pt,a4paper]{article}

\usepackage{amsmath,amssymb,amsthm,pb-diagram}
\usepackage{enumerate}
\usepackage{amsmath}
\usepackage{amsfonts}
\usepackage{authblk}
\usepackage{csquotes}
\usepackage[table]{xcolor}% http://ctan.org/pkg/xcolor
\usepackage{graphicx}
\usepackage{verbatim}

\usepackage[left=1.25in, right=1.25in, top=1.25in, bottom=1.25in]{geometry}

\usepackage{graphicx}%
\usepackage{multirow}%
\usepackage{amsmath,amssymb,amsfonts}%
\usepackage{amsthm}%
\usepackage{mathrsfs}%
\usepackage[title]{appendix}%
\usepackage{xcolor}%
\usepackage{textcomp}%
\usepackage{manyfoot}%
\usepackage{booktabs}%
\usepackage{algorithm}%
\usepackage{algorithmicx}%
\usepackage{algpseudocode}%
\usepackage{listings}%

\numberwithin{equation}{section}

\theoremstyle{plain}
\newtheorem{theorem}{Theorem}[section]
\newtheorem{lemma}[theorem]{Lemma}

\newtheorem{proposition}[theorem]{Proposition}

\theoremstyle{definition}

\newtheorem{remark}{Remark}[section]

\newcommand{\id}{\mathrm{id}}

\begin{document}

\begin{center}{\large\textbf{REGULARITY OF THE SEMIGROUP OF TRANSFORMATIONS PRESERVING A LENGTH}}
\end{center}

\vskip 0.5cm
 \begin{center} Worachead Sommanee \end{center} 
\begin{center} Department of Mathematics and Statistics,  Chiang Mai Rajabhat University, \newline Chiang Mai, 50300, Thailand
\end{center}  
\begin{center} E-mail: worachead\_som@cmru.ac.th
\end{center} 

\vskip0.3cm
\begin{abstract}
	Let $X_n = \{1,2,\dots,n\}$ be a finite set $(n\geq 2)$ and $T_n$ the full transformation semigroup on $X_n$. For a positive integer $l\leq n-1$, we define 
	\begin{center}
		$T_n(l) = \{\alpha\in T_n \colon \forall x,y\in X_n,\, |x-y| = l \;\Rightarrow\; |x\alpha - y\alpha| = l\}$	
	\end{center}
	and
	\begin{center}
		$T^*_n(l) = \{\alpha\in T_n \colon \forall x,y\in X_n,\, |x-y| = l \;\Leftrightarrow\; |x\alpha - y\alpha| = l\}$.	
	\end{center}
	Then $T_n(l)$ and $T^*_n(l)$ are subsemigroups of $T_n$. In this paper, we  give a necessary and sufficient condition for $T_n(l)$ to be regular. Moreover, we prove that  $T^*_n(l)$ is a regular semigroup.	
\end{abstract}



\section{Introduction and Preliminaries}\label{intro}
%\lipsum[2-4]

Regularity has played an important role in the development of semigroup theory. For a concept of regularity and transformation semigroups,  we refer the reader to \cite{CliffordPreston1961}.

The \emph{full transformation semigroup} is the collection of functions from a set $X$ into $X$ with composition, which is denoted by $T(X)$. It is well-known that $T(X)$ is a regular semigroup with identity element $\id_X$, the identity function on $X$ (see \cite{Doss}). The regularity of various subsemigroups of $T(X)$ have been studied over a long period (see, for example \cite{Chinram, Jitman, Mendes, Palasak, Pookpienlert, Sangkhanan, Srithus, Zhao}).

For $x\in X$ and $\alpha\in T(X)$, the image of $x$ under $\alpha$ is written as $x\alpha$. By the corresponding notation, we will compose functions from left to right: $x(\alpha\beta) = (x\alpha)\beta$.  For a non-empty subset $A$ of $X$, the image of a subset $A$ of $X$ under $\alpha$ is denoted by $A\alpha$. If $A=X$, then $X\alpha$ is the range (image) of $\alpha$. We denote by $x\alpha^{-1}$ the set of all inverse images of $x$ under $\alpha$, that is, $x\alpha^{-1} = \{z\in X \colon z\alpha = x\}$.  Here, $|A|$ denotes the cardinality of a set $A$.

Throughout the paper, we let $X_n = \{1,2,\dots,n\}$ $(n\geq 2)$ and $l$ an integer such that $1\leq l\leq n-1$. We write $T_n$ instead of $T(X_n)$. Define 
\begin{center}
	$
	T_n(l) = \{\alpha\in T_n \colon \forall x,y\in X_n,\, |x-y| = l \;\Rightarrow\; |x\alpha - y\alpha| = l\}
	$
\end{center} 
and 
\begin{center}
	$
	T^*_n(l) = \{\alpha\in T_n \colon \forall x,y\in X_n,\, |x-y| = l \;\Leftrightarrow\; |x\alpha - y\alpha| = l\}
	$
\end{center} 
\vskip0.1cm
\noindent where $|x-y|$ is the \emph{absolute difference} of numbers $x$ and $y$. Obviously, $\id_{X_n}\in T^*_n(l)$ and $T^*_n(l)\subseteq T_n(l)$.  Note that if $\alpha\in T_n(l)$, then we say that $\alpha$ \emph{preserves the length} $l$. Let $\alpha,\beta\in T_n(l)$ and $x,y\in X_n$ such that $|x-y| = l$. Then $|x\alpha - y\alpha| = l$. Since $x\alpha,y\alpha\in X_n$ and $\beta$ preserves the length $l$, we obtain $|(x\alpha)\beta - (y\alpha)\beta| = l$. Hence $\alpha\beta\in T_n(l)$ and so $T_n(l)$ is a subsemigroup of $T_n$. We call $T_n(l)$ the \emph{semigroup of transformations preserving a length} $l$. Similarly, if $\alpha,\beta\in T^*_n(l)$, we can show that
\vskip0.2cm
\begin{center}
	$|x(\alpha\beta) - y(\alpha\beta)| = l \Leftrightarrow |(x\alpha)\beta - (y\alpha)\beta| = l \Leftrightarrow |x\alpha - y\alpha| = l \Leftrightarrow |x - y| = l$.
\end{center}
\vskip0.2cm 
\noindent Hence $\alpha\beta\in T^*_n(l)$ and so $T^*_n(l)$ is a subsemigroup of $T_n$. It is also clear that $T^*_n(l)$ is a subsemigroup of $T_n(l)$. 

Here, in Section \ref{intro}, we present the new transformation semigroups $T_n(l)$ and $T^*_n(l)$. Moreover, we also introduce basic results for the semigroups $T_n(l)$ and $T^*_n(l)$. In Sections \ref{Regularity1} and \ref{Regularity=l}, we give a necessary and sufficient condition for $T_n(l)$ to be regular.  Finally, in Section \ref{Reg_T_star}, we prove that $T^*_n(l)$ is a regular semigroup.

As in Clifford and Preston \cite[page 2]{CliffordPreston1961}, if $a_1,a_2,\dots, a_n\in X_n$ (not necessarily distinct), we shall use the notation
\begin{center}
	$
	\alpha=\begin{pmatrix}
		1 & 2 & \cdots & n\\
		a_1 & a_2 & \cdots & a_n
	\end{pmatrix}$   
\end{center}
\vskip0.1cm
to mean $\alpha\in T_n$ defined by $i\alpha = a_i$ for all $i = 1,2,\dots,n$.

The following remark will be useful in the paper.

\begin{remark}\label{RemK} The following hold:
	\begin{enumerate}
		\item  If $(n\geq 3$ is odd and $l \geq \frac{n+1}{2})$ or $(n \geq 4$ is even and  $l > \frac{n}{2})$, then $n-l < l$. For each $x\in\{1,2,\dots,n-l\}$, there exists unique $y = x+l\in X_n$ such that $|x - y| = l$. On the other hand, for each $x \in\{l+1,l+2,\dots,n\}$, there exists unique $y = x-l\in X_n$ such that $|x - y| = l$. Moreover, $|x-y|\neq l$ for all  $x\in\{(n-l)+1,(n-l)+2,\dots, l\}$ and all $y\in X_n$.  Indeed, 
		for $1\leq k \leq 2l-n$, if $|[(n-l)+k] - y| = l$ for some $y\in X_n$, we get that $[(n-l)+k] - y = l$ or $[(n-l)+k] - y = -l$. In the former, we obtain $n-2l+k = y > 0$, which implies that $k > 2l-n$, a contradiction. In the latter, we have $y = n+k \notin X_n$. 	So in this case, we can write
		\begin{center}
			$X_n = \{1, 2, \dots, n-l, (n-l)+1, \dots, l-1, l, l+1, l+2, \dots, n\}$
		\end{center} or
		\begin{center}
			$X_n = \{1, l+1\}\cup\{2, l+2\}\cup\cdots\cup\{n-l, n\}\cup\{(n-l)+1, (n-l)+2, \dots, l-1, l\}$.
		\end{center}
		
		\item If $(n\geq 3$ is odd and $l < \frac{n+1}{2})$ or $(n \geq 4$ is even and  $l \leq \frac{n}{2})$, then $2l \leq n$.  So, we can write 
		\begin{center}
			$X_n =\{1, 2, \dots, l, l+1, l+2,  \dots, 2l, 2l+1, 2l+2, \dots, 3l, 3l+1, 3l+2,  \dots, n\}$
		\end{center}
		or
	
	\hskip0.6cm 	$
	X_n = \{1,1+l,1+2l,\dots,1+m_1l\}\cup \{2,2+l,2+2l,\dots,2+m_2l\}\cup \cdots\cup$ 
	
	\hskip1.65cm $\{l,2l, 3l,\dots, (1+m_l)l\}	
	$
			
		where $m_i\geq 1$ is the maximum integer such that $i+m_il\leq n$ for all $i = 1,2,\dots,l$. It is clear that for each $x\in X_n$, there exists $y\in X_n$ such that $|x-y|=l$.
	\end{enumerate}
\end{remark}

\begin{proposition}\label{Property1} $T_n(l)$ is never equal to $T_n$.
\end{proposition}

\begin{proof} Let $\kappa\in T_n$ be the constant function with range $\{1\}$. Then $\kappa\notin T_n(l)$ since $|1-(l+1)| = l$ but $|1\kappa-(l+1)\kappa| = |1-1| = 0 \neq l$. Hence, $T_n(l)\neq T_n$. 
\end{proof} 

%
%\vskip0.2cm
%Notice that $|X_n\alpha|\geq 2$ for all $\alpha\in T_n(l)$ since $T_n(l)$ does not contain a constant map.

\begin{proposition}\label{Property2} $T_n(l) = T^*_n(l)$ if and only if $n=2$ or $(n=3$ and $l=1)$.
\end{proposition}

\begin{proof} If $n=2$, then $l=1$ and 
	\begin{center} $T_2(1) = \biggl\{ 
		\begin{pmatrix}
			1 & 2\\
			1 & 2 
		\end{pmatrix} , 	\begin{pmatrix}
			1 & 2\\
			2 & 1 
		\end{pmatrix} 
		\biggl\} = T^*_2(1)$.
	\end{center}
	If $n=3$ and $l=1$, then 

\begin{align*}
T_3(1) &= \biggl\{ 
\scriptstyle{\begin{pmatrix}
	1 & 2 & 3\\
	1 & 2 & 1
\end{pmatrix} , 	\begin{pmatrix}
	1 & 2 & 3\\
	1 & 2 & 3
\end{pmatrix} ,
\begin{pmatrix}
	1 & 2 & 3\\
	2 & 1 & 2
\end{pmatrix} , 	\begin{pmatrix}
	1 & 2 & 3\\
	2 & 3 & 2
\end{pmatrix},
\begin{pmatrix}
	1 & 2 & 3\\
	3 & 2 & 1
\end{pmatrix} , 	\begin{pmatrix}
	1 & 2 & 3\\
	3 & 2 & 3
\end{pmatrix}}
\biggl\}\\
&=T^*_3(1).
\end{align*}	
Conversely, assume  $n\geq 4$ or ($n =3$ and $l = 2$). We show $T_n(l) \neq T^*_n(l)$. Let $\alpha = \begin{pmatrix}
		1 & 2 & 3\\
		1 & 1 & 3 
	\end{pmatrix}$. Then $\alpha\in T_3(2)\setminus T^*_3(2)$, since $|2\alpha - 3\alpha| = |1-3| = 2$ but $|2-3| = 1 \neq 2$. Thus, $T_3(2) \neq T^*_3(2)$. Now, suppose that $n\geq 4$.  We consider four cases: (i) $n$ is even and  $l > \frac{n}{2}$; (ii) $n$ is odd and $l \geq \frac{n+1}{2}$; (iii) $n$ is even and $l \leq \frac{n}{2}$; (iv) $n$ is odd and $l < \frac{n+1}{2}$.

	\noindent\textbf{Cases (i) \& (ii)}.  Using  Remark \ref{RemK}\,(1), we define $\alpha\in T_n(l)$ by
	\begin{center}
		$\small{\alpha = \scriptscriptstyle{\left( {\begin{array}{ccccccccccccccccccccc}
					1 & l+1 & 2 & l+2 & \cdots &  n-l & n  & (n-l)+1 & (n-l)+2 & \cdots& l-1 & l\\
					1 & l+1 & 1 & l+1 & \cdots &  1 & l+1  & 1 & 1 & \cdots& 1 & 1
			\end{array} } \right)}}$. 
	\end{center}
	We see that $|l\alpha - (l+1)\alpha| = |1-(l+1)| = l$ while $|l-(l+1)| = 1 \neq l$ (since $l > \frac{n}{2}\geq 2$). Hence $\alpha\notin T^*_n(l)$.
	
	\noindent\textbf{Cases (iii) \& (iv)}.  Using  Remark \ref{RemK}\,(2), we can define $\alpha\in T_n(1)$ by
	\begin{center}
		$\alpha = \left( {\begin{array}{ccccccccccccccccccccc}
				1 & 2 & 3 & 4 & \cdots &  n\\
				1 & 2 & 1 & 2 & \cdots & n\alpha
		\end{array} } \right)$ 
	\end{center}
	where \hskip 1.25cm $n\alpha=\begin{cases}
		1            & \hbox{$\text{if } n \text{ is odd}$},\\
		2            & \hbox{$\text{if } n \text{ is even}$}.
	\end{cases}$
	\vskip0.05cm
	\noindent Since $|1\alpha - 4\alpha| = |1-2| = 1$ but $|1-4| = 3\neq 1$, thus $\alpha\notin T^*_n(l)$. For $l\geq 2$, we define $\alpha\in T_n(l)$ by
	\begin{center}
		$\small{\alpha = \footnotesize{\left( {\begin{array}{cccccccccccccccccccccccc}
					1  & \cdots & l & l+1  &  \cdots & 2l & 2l+1   & \cdots & 3l & 3l+1 & \cdots & 4l & 4l+1 &\cdots & n\\
					1  & \cdots & 1 & l+1  &  \cdots & l+1 & 1  & \cdots & 1 & l+1 &  \cdots & l+1 & 1 & \cdots & n\alpha
			\end{array} } \right)}}$ 
	\end{center}
	where 
	\begin{center}
		$n\alpha=\begin{cases}
			1            & \hbox{$\text{if } n\in\{kl+t\colon k\geq 2 \text{ is even and } 1\leq t \leq l\} $},\\
			l+1            & \hbox{$\text{if } n\in\{kl+t\colon k\geq 3 \text{ is odd and } 1\leq t \leq l\} $}.
		\end{cases}$.
	\end{center}
	It is easy to see that $|2\alpha - (l+1)\alpha| = |1-(l+1)| = l$ such that $|2-(l+1)| = |1-l|\neq l$. Hence $\alpha\notin T^*_n(l)$. Therefore, $T_n(l)\neq T^*_n(l)$.
\end{proof}

%%%%%%%%%%%%%%%%%% Section 2 Regularity %%%%%%%%%%%%%%%%%%%%%%%%%%%%%%

\section{Regularity of $T_n(1)$}\label{Regularity1}

We would like to acknowledge the use of GAP \cite{GAP}, a system for
computational discrete algebra with the MONOID package.

\begin{remark}\label{Remm1} Recall that the subsemigroup generated by $A$ denoted by $\langle A\rangle$. Using the commands gap> Semigroup(); and gap> IsRegularSemigroup(); of GAP, we have the following facts.
	
	\noindent  (1) $T_2(1) = \biggl\langle
	\begin{pmatrix}
		1 & 2\\
		2 & 1 
	\end{pmatrix} 
	\biggl\rangle
	$ is regular,
	
	\noindent  (2) $T_3(1) = \biggl\langle
	\begin{pmatrix}
		1 & 2  & 3\\
		2 & 1  & 2
	\end{pmatrix},  
	\begin{pmatrix}
		1 & 2  & 3\\
		3 & 2  & 1
	\end{pmatrix} \biggl\rangle
	$ is regular,

	\noindent  (3) $T_4(1) = \biggl\langle
	\begin{pmatrix}
		1 & 2  & 3 & 4\\
		2 & 3  & 2 & 1
	\end{pmatrix},  
	\begin{pmatrix}
		1 & 2  & 3 & 4\\
		3 & 4  & 3 & 2
	\end{pmatrix},
	\begin{pmatrix}
		1 & 2  & 3 & 4\\
		4 & 3  & 2 & 1
	\end{pmatrix} \biggl\rangle
	$ is regular,
	
	\noindent  (4) $\footnotesize{T_5(1) = \biggl\langle
		\begin{pmatrix}
			1 & 2  & 3 & 4 & 5\\
			2 & 3 & 4 & 5 & 4 
		\end{pmatrix},  
		\begin{pmatrix}
			1 & 2  & 3 & 4 & 5\\
			3 & 2  & 1 & 2 & 3
		\end{pmatrix},
		\begin{pmatrix}
			1 & 2  & 3 & 4 & 5\\
			4 & 3  & 2 & 1 & 2
		\end{pmatrix},
		\begin{pmatrix}
			1 & 2  & 3 & 4 & 5\\
			5 & 4  & 3 & 2 & 1
		\end{pmatrix}
		\biggl\rangle}$
	
	\hskip1.4cm	 is regular.
\end{remark}

\begin{theorem}\label{Regular1} $T_n(1)$ is a regular semigroup if and only if $n\in\{2,3,4,5\}$.
\end{theorem}

\begin{proof} By Remark \ref{Remm1}, we have $T_n(1)$ is a regular semigroup when $n\in\{2,3,4,5\}$. 
	
	Conversely, let $n\geq 6$ and assume that $T_n(1)$ is regular. There are four cases to consider.
	
	\noindent\textbf{Case 1.} Suppose that $n\in\{6,10,14,18,\dots\}$. Then $|\{1, 2, 3, \dots, \frac{n}{2}\}| = |\{n, n-1, \dots,  \break  \frac{n}{2}+1\}| = \frac{n}{2}$ is odd.  Define 
	\vskip0.1cm
	\begin{center}
		$
		\footnotesize{\alpha = \left( {\begin{array}{cccccccccccccc}
					1 & 2 & \cdots & \frac{n}{2}-1 & \frac{n}{2} & \frac{n}{2}+1 & \frac{n}{2}+2 & \frac{n}{2}+3 & \frac{n}{2}+4 & \cdots & n-2  & n-1 & n\\
					n & n-1 & \cdots   & \frac{n}{2}+2 & \frac{n}{2}+1 & \frac{n}{2}+2 & \frac{n}{2}+1 & \frac{n}{2}+2 & \frac{n}{2}+1 & \cdots &  \frac{n}{2}+2   & \frac{n}{2}+1  & \frac{n}{2} 
			\end{array} } \right)} 
		$.	
	\end{center}
	\noindent It is easy to check that $\alpha\in T_n(1)$. Since $T_n(1)$ is regular, there exists $\beta\in T_n(1)$ such that $\alpha = \alpha\beta\alpha$. We have $\frac{n}{2} = n\alpha = (n\alpha\beta)\alpha = (\frac{n}{2}\beta)\alpha$.  Then we obtain $(\frac{n}{2})\beta = n$ and it follows that $(\frac{n}{2}+1)\beta = n-1$.  Since $\frac{n}{2}+2 = (\frac{n}{2}+1)\alpha = ((\frac{n}{2}+1)\alpha\beta)\alpha = ((\frac{n}{2}+2)\beta)\alpha$, we get $(\frac{n}{2}+2)\beta\in\{\frac{n}{2}-1, \frac{n}{2}+1, \frac{n}{2}+3, \dots, n-2\}$. We see that $(\frac{n}{2}+2)\beta \in\{n,n-2\}$, so  $(\frac{n}{2}+2)\beta = n-2$. Next, we consider $\frac{n}{2}+3 = (\frac{n}{2}-2)\alpha = ((\frac{n}{2}-2)\alpha\beta)\alpha = ((\frac{n}{2}+3)\beta)\alpha$. Thus $(\frac{n}{2}+3)\beta = \frac{n}{2}-2$. Since $|(\frac{n}{2}+3) - (\frac{n}{2}+2)| = 1$ but $|(\frac{n}{2}+3)\beta - (\frac{n}{2}+2)\beta| = |(\frac{n}{2}-2) - (n-2)| = \frac{n}{2} \neq 1$, this is a contradiction.
	
	\noindent\textbf{Case 2.} Suppose that $n\in\{8, 12, 16, 20, \dots\}$. We note that  $|\{\frac{n}{2}+1, \frac{n}{2}+2, \dots, \break n-3, n-2, n-1, n\}| = \frac{n}{2}$ is even. Define 
	\begin{center}
		$
		\footnotesize{\alpha = \left( {\begin{array}{ccccccccccccccc}
					1 & 2 & \cdots  & \frac{n}{2}-1 & \frac{n}{2} & \frac{n}{2}+1 & \frac{n}{2}+2  & \cdots & n-3 & n-2 & n-1 & n\\
					n & n-1 & \cdots  & \frac{n}{2}+2 & \frac{n}{2}+1  & \frac{n}{2}+2  & \frac{n}{2}+1 & \cdots & \frac{n}{2}+2  & \frac{n}{2}+1  & \frac{n}{2}  & \frac{n}{2}-1 
			\end{array} } \right)}
		$.	
	\end{center}
	It is obvious that $\alpha\in T_n(1)$. Since $T_n(1)$ is regular, there exists $\beta\in T_n(1)$ such that $\alpha = \alpha\beta\alpha$. Consider $\frac{n}{2} = (n-1)\alpha = ((n-1)\alpha\beta)\alpha = (\frac{n}{2}\beta)\alpha$.  Then we obtain $\frac{n}{2}\beta = n-1$.  Since $\frac{n}{2}+1 = (n-2)\alpha = ((n-2)\alpha\beta)\alpha = ((\frac{n}{2}+1)\beta)\alpha$, we get  $(\frac{n}{2}+1)\beta\in\{\frac{n}{2}, \frac{n}{2}+2, \break \frac{n}{2}+4, \dots,n-4, n-2\}$ and so $(\frac{n}{2}+1)\beta = n-2$. Since $((\frac{n}{2}+2)\beta)\alpha = ((n-3)\alpha)\beta\alpha = (n-3)\alpha = \frac{n}{2}+2$, which implies that $(\frac{n}{2}+2)\beta\in\{\frac{n}{2}-1, \frac{n}{2}+1, \frac{n}{2}+3, \dots, n-3\}$. But $(\frac{n}{2}+2)\beta\in\{n-1, n-3\}$, whence $(\frac{n}{2}+2)\beta =  n-3$. Next, we consider $\frac{n}{2}+3 = (\frac{n}{2}-2)\alpha = ((\frac{n}{2}-2)\alpha\beta)\alpha = ((\frac{n}{2}+3)\beta)\alpha$. Thus $(\frac{n}{2}+3)\beta = \frac{n}{2}-2$. We see that $|(\frac{n}{2}+3) - (\frac{n}{2}+2)| = 1$ but $|(\frac{n}{2}+3)\beta - (\frac{n}{2}+2)\beta| = |(\frac{n}{2}-2) - (n-3)| = \frac{n}{2} -1 \neq 1$. This is a contradiction.
	
	\noindent\textbf{Case 3.} Suppose that $n\in\{7,11,15,19,\dots\}$. 	Then  $|\{2, 3, \dots,\frac{n+1}{2}\}| = |\{n, n-1, \break n-2,  \dots, \frac{n+5}{2}, \frac{n+3}{2}\}|$ and $|\{\frac{n+3}{2}, \frac{n+5}{2}, \dots, n-2, n-1\}| = \frac{n-3}{2}$ is even. Define 		
	\begin{center}
		$
		\footnotesize{\alpha = \left( {\begin{array}{cccccccccccccc}
					1 & 2 & 3 & \cdots & \frac{n-3}{2} & \frac{n-1}{2} & \frac{n+1}{2} & \frac{n+3}{2} & \frac{n+5}{2} & \cdots  & n-2 & n-1 & n\\
					n-1 & n & n-1 & \cdots  & \frac{n+7}{2} & \frac{n+5}{2} & \frac{n+3}{2} & \frac{n+5}{2} & \frac{n+3}{2} & \cdots  & \frac{n+5}{2}  & \frac{n+3}{2}  & \frac{n+1}{2} 
			\end{array} } \right)}
		$.
	\end{center}
	It is easy to see that $\alpha\in T_n(1)$. Since $T_n(1)$ is regular, there exists $\beta\in T_n(1)$ such that $\alpha = \alpha\beta\alpha$. Consider $\frac{n+1}{2} = n\alpha = (n\alpha\beta)\alpha = ((\frac{n+1}{2})\beta)\alpha$.  Then we obtain $(\frac{n+1}{2})\beta = n$ and it follows that $(\frac{n+3}{2})\beta = n-1$.  Since $\frac{n+5}{2} = (\frac{n+3}{2})\alpha = ((\frac{n+3}{2})\alpha\beta)\alpha = ((\frac{n+5}{2})\beta)\alpha$, we get	$(\frac{n+5}{2})\beta\in\{\frac{n-1}{2}, \frac{n+3}{2}, \frac{n+7}{2}, \dots, n-2\}$ and so $(\frac{n+5}{2})\beta = n-2$. Next, we consider $\frac{n+7}{2} = (\frac{n-3}{2})\alpha = ((\frac{n-3}{2})\alpha\beta)\alpha = ((\frac{n+7}{2})\beta)\alpha$. Thus $(\frac{n+7}{2})\beta = \frac{n-3}{2}$. We see that $|\frac{n+7}{2} - \frac{n+5}{2}| = 1$ but $|(\frac{n+7}{2})\beta - (\frac{n+5}{2})\beta| = |\frac{n-3}{2} - (n-2)| = \frac{n-1}{2} \neq 1$. This is a contradiction.

	\noindent\textbf{Case 4.} Suppose that $n\in\{9,13,17,21,\dots\}$. Define 		
	\vskip  0.1cm
	\begin{center}
		$
		\footnotesize{\alpha = \left( {\begin{array}{cccccccccccccc}
					1 & 2 & 3 & \cdots & \frac{n-5}{2} & \frac{n-3}{2} & \frac{n-1}{2} & \frac{n+1}{2} & \frac{n+3}{2} & \cdots  & n-2 & n-1 & n\\
					n & n-1 & n-2 & \cdots  & \frac{n+7}{2} & \frac{n+5}{2} & \frac{n+3}{2} & \frac{n+5}{2} & \frac{n+3}{2} &  \cdots  & \frac{n+5}{2}  & \frac{n+3}{2}  & \frac{n+1}{2} 
			\end{array} } \right)}
		$.
	\end{center}
	We note that $|\{1, 2, \dots,\frac{n-3}{2}, \frac{n-1}{2}\}| = |\{n, n-1,  \dots, \frac{n+5}{2}, \frac{n+3}{2}\}|$ and $|\{\frac{n+1}{2}, \frac{n+3}{2}, \dots, \break n-2, n-1\}| = \frac{n-1}{2}$ is even.  It is clear that $\alpha\in T_n(1)$. Since $T_n(1)$ is regular, there exists $\beta\in T_n(1)$ such that $\alpha = \alpha\beta\alpha$. Consider $\frac{n+1}{2} = n\alpha = (n\alpha\beta)\alpha = ((\frac{n+1}{2})\beta)\alpha$.  Then we obtain $(\frac{n+1}{2})\beta = n$ and it follows that $(\frac{n+3}{2})\beta = n-1$.  Since $\frac{n+5}{2} = (n-2)\alpha = \break  ((n-2)\alpha\beta)\alpha = ((\frac{n+5}{2})\beta)\alpha$, we get	$(\frac{n+5}{2})\beta\in\{\frac{n-3}{2}, \frac{n+1}{2}, \frac{n+5}{2}, \dots, n-2\}$ and so $(\frac{n+5}{2})\beta = n-2$. Next, we consider $\frac{n+7}{2} = (\frac{n-5}{2})\alpha = ((\frac{n-5}{2})\alpha\beta)\alpha = ((\frac{n+7}{2})\beta)\alpha$. Thus $(\frac{n+7}{2})\beta = \frac{n-5}{2}$. Since $|\frac{n+7}{2} - \frac{n+5}{2}| = 1$, this implies that                    $|(\frac{n+7}{2})\beta - (\frac{n+5}{2})\beta| = |\frac{n-5}{2} - (n-2)| = \frac{n+1}{2} \neq 1$. This is a contradiction. Therefore 	$T_n(1)$ is not regular. 
\end{proof}

%%%%%%%%%%%%%%%%%%% Section 3 Regularity %%%%%%%%%%%%%%%%%%%%%%%%%%%%%%

\section{Regularity of $T_n(l)$ for $2\leq l\leq n-1$}\label{Regularity=l}

For the case $l=n-1$, we have 
\begin{align*}
	T_n(n-1)    &=\{\alpha\in T_n \colon 1\alpha = 1 \text{ and } n\alpha = n\}\cup\{\alpha\in T_n \colon 1\alpha = n \text{ and } n\alpha = 1\}\\
	&= PG_{\{1,n\}}(X_n) \;\text{ as defined in } \cite{Laysirikul}.
\end{align*}
It follows from \cite[Theorem 2.2]{Laysirikul} that $T_n(n-1)$ is a regular semigroup.

From now on, we let $n\geq 4$ and $2\leq l \leq n-2$. Note that $2 \leq n-l$ and $l+2\leq n$.

\begin{lemma}\label{n_odd_NotReg} If $n\geq 5$ is odd, then  $T_n(l)$ is not regular.
\end{lemma}
\begin{proof} Let $n\geq 5$ be an odd integer. We consider two cases: $l \geq \frac{n+1}{2}$ or $l < \frac{n+1}{2}$. 
	
	\noindent\textbf{Case 1.} Assume that $l \geq \frac{n+1}{2}$.  Using Remark \ref{RemK}\,(1), we can define
	\begin{center}
		$\small{\alpha =  \left({\begin{array}{ccccccccccccccccccccc}
				1 & 2 & \cdots & n-l & (n-l)+1 &  \cdots & l-1  & l & l+1 & l+2 & \cdots & n\\
				1 & 1 & \cdots & 1 & 1 &  \cdots & 1 & 2 & l+1 & l+1 & \cdots & l+1
		\end{array} } \right)}$.
	\end{center} 
	It is easy to verify that $\alpha\in T_n(l)$. If $T_n(l)$ is regular, then there exists $\beta\in T_n(l)$ such that $\alpha = \alpha\beta\alpha$. It follows that $2 = l\alpha = (l\alpha\beta)\alpha = (2\beta)\alpha$ and so $2\beta = l$. Since $|2-(l+2)| = l$, this implies that $|l-(l+2)\beta| = |2\beta-(l+2)\beta| = l$. Hence $(l+2)\beta = 0$ or $(l+2)\beta = 2l > n$, this is a contradiction. Therefore, $T_n(l)$ is not regular.
	
	\noindent\textbf{Case 2.} Assume that $l < \frac{n+1}{2}$. Then $l \leq \frac{n+1}{2} -1 = \frac{n-1}{2}$, whence $2l+1\leq n$. Using Remark \ref{RemK}\,(2), we can define $\alpha\in T_n(l)$ by
	\begin{center}
		$
		\small{\alpha = \footnotesize{\left( {\begin{array}{ccccccccccccccccccccc}
					1 & 2 & 3 & \cdots & l+1 & l+2 & l+3 &  \cdots & kl+1 & kl+2  & kl+3 & \cdots & n\\
					l+1 & l+1 & 3 & \cdots & 2l+1 & 1 & l+3 &  \cdots & u_k & v_k & kl+3 &  \cdots & n\alpha
			\end{array} } \right)}} 
		$
	\end{center}
	where for $k\geq 2$,
	\begin{center}
		$
		u_k = (kl+1)\alpha=\begin{cases}
			(k+1)l+1            & \hbox{$\text{if } (k+1)l+1\leq n$},\\
			(k-1)l+1             & \hbox{$\text{if } (k+1)l+1>n$},
		\end{cases}
		$
	\end{center}
	\begin{center}
		$
		v_k = (kl+2)\alpha=\begin{cases}
			1              & \hbox{$\text{if } k \text{ is odd}$},\\
			l+1            & \hbox{$\text{if } k \text{ is even}$},
		\end{cases}
		$ \:\:and \:\: 	$n\alpha=\begin{cases}
			n-l          & \hbox{$\text{if } r= 0$},\\
			1            & \hbox{$\text{if } r = 1$},\\
			n            & \hbox{$\text{if } 2\leq r \leq l-1$}
		\end{cases}$
	\end{center}
	when $r$ is the remainder of a division of $n-1$ by $l$.  Note that if $n\alpha = 1$ (that is, $r=1$), then $n = kl+2$ for some a positive odd integer  $k$. Suppose that there is an element $\beta\in T_n(l)$ such that $\alpha = \alpha\beta\alpha$. Consider $1 = (l+2)\alpha = ((l+2)\alpha)\beta\alpha = (1\beta)\alpha$ and $2l+1 = (l+1)\alpha = ((l+1)\alpha)\beta\alpha = ((2l+1)\beta)\alpha$. By the definition of $\alpha$, we obtain $1\beta = kl+2$ for some $k\in\{1,3,5,7,\dots\}$ and $(2l+1)\beta\in\{l+1,3l+1\}$. Let $(l+1)\beta = z\in X_n$.  Assume that $(2l+1)\beta = l+1$. Since $|(l+1) - 1| = l$ and $|(2l+1) - (l+1)| = l$, whence $|(l+1)\beta - 1\beta| = l$ and $|(2l+1)\beta - (l+1)\beta| = l$. That is, $|z -  (kl+2)| = l$ and $|l+1 - z| = l$ where $k\geq 1$ is odd. Thus there are four possible cases to consider: 
	\begin{enumerate}[(i)]
		\item $z -  (kl+2) = l$ and $l+1 - z = l$, 
		
		\item $z -  (kl+2) = l$ and $l+1 - z = -l$, 
		
		\item  $z -  (kl+2) = -l$ and $l+1 - z = l$,

		\item $z -  (kl+2) = -l$ and $l+1 - z = -l$. 
	\end{enumerate}

	\noindent For the cases (i) and (iv), we get that $2z = (k+1)l+3$, this is a contradiction since $2z$ is even while $(k+1)l+3$ is odd. For the cases (ii) and (iii), we obtain $(k-1)l = -1$, this is a contradiction since $(k-1)l$ is  even. And if $(2l+1)\beta = 3l+1$, then a contradiction is obtained using an analogous argument. Therefore, $\alpha$ is not a regular element of $T_n(l)$.  
\end{proof}

\begin{lemma}\label{n_even_NotReg} If $n\geq 6$ is even such that $l\neq \frac{n}{2}$, then  $T_n(l)$ is not regular.
\end{lemma}
\begin{proof} Assume that $n\geq 6$ is even such that $l\neq \frac{n}{2}$. If  $l > \frac{n}{2}$, we let $\alpha$ be as defined in \textbf{Case 1} of Lemma \ref{n_odd_NotReg}. By using the same proof, we obtain $\alpha$ is not a regular element in $T_n(l)$. Next, suppose that $l < \frac{n}{2}$. Then $2l+1 \leq n$.  Let $u_k$ and $v_k$ be as defined in \textbf{Case 2} of Lemma \ref{n_odd_NotReg}. Define  $\alpha\in T_n(l)$ by
	\begin{center}
		$\small{\alpha = \footnotesize{\left( {\begin{array}{ccccccccccccccccccccc}
					1 & 2 & 3 & \cdots & l+1 & l+2 & l+3 &  \cdots & kl+1 & kl+2  & kl+3 & \cdots & n\\
					l+1 & l+1 & 3 & \cdots & 2l+1 & 1 & l+3 &  \cdots & u_k & v_k & kl+3 &  \cdots & n\alpha
			\end{array} } \right) }}
		$
		
	\end{center}
	where  
	\begin{center}
		$n\alpha=\begin{cases}
			n-l          & \hbox{$\text{if } r= 0$},\\
			1            & \hbox{$\text{if } q \text{ is odd and } r = 1$},\\
			l+1            & \hbox{$\text{if } q \text{ is even and } r = 1$},\\
			n            & \hbox{$\text{if } 2\leq r \leq l-1$}
		\end{cases}$
	\end{center}
	such that $q$ is the quotient and $r$ is the remainder of a division of $n-1$ by $l$. Note that if $n\alpha = 1$ (that is, $q$ is odd and $r = 1$), then $n = ql + 2$ where $q$ is odd. By the same proof as given for \textbf{Case 2} of Lemma \ref{n_odd_NotReg}, we obtain $\alpha$ is not a regular element in $T_n(l)$.
\end{proof}

\begin{remark}\label{Rem2} If $n$ is even and $l = \frac{n}{2}$, then $n=2l$ and $n-l = l$. We note that $x+l\notin X_n$ for all $x\in\{l+1,l+2,\dots,n\}$. Of importance is the fact that for each $x\in\{1,2,\dots,l\}$, there exists unique $y = x+l\in X_n$ such that $|x - y| = l$. And, for each $x \in\{l+1,l+2,\dots,n\}$, there exists unique $y = x-l\in X_n$ such that $|x - y| = l$. Moreover, for any $\alpha\in T_n(\frac{n}{2})$ we can write
	\begin{center}
		$		\alpha = \left( {\begin{array}{ccccccccccccccccccccc}
				1 & l+1 & 2 &  l+2 & \cdots & l & 2l  \\
				1\alpha & (l+1)\alpha & 2\alpha &  (l+2)\alpha & \cdots & l\alpha & (2l)\alpha 
		\end{array} } \right) 
		$
	\end{center}
	where for each $1\leq x\leq l$, 
	\begin{center}		
		$(x\alpha, (l+x)\alpha) \in \{(i, l+i)\colon 1\leq i\leq l\} \cup \{(l+i, i)\colon 1\leq i\leq l\}$.
	\end{center}
\end{remark}

\begin{theorem}\label{Regular4} Let $n\geq 4$ and $2\leq l \leq n-2$. Then
	$T_n(l)$ is regular if and only if $n$ is even and $l=\frac{n}{2}$.
\end{theorem}

\begin{proof} It follows from Lemmas \ref{n_odd_NotReg},  \ref{n_even_NotReg} that $T_n(l)$ is not regular when $n$ is odd or $l\neq \frac{n}{2}$. Conversely, assume that $n$ is even and $l=\frac{n}{2}$. We prove $T_n(l)$ is regular. Let $\alpha$ be any element in $T_n(l)$ and $u\in X_n\alpha$. By Remark \ref{Rem2}, there exists unique $v\in X_n\alpha$ such that $|u-v| = l$. In fact, $(u, v) \in \{(i, l+i), (l+i, i) \colon 1\leq i\leq l\}$.  We choose and fix $d_u\in u\alpha^{-1}$ and $d_v\in v\alpha^{-1}$ such that $|d_u - d_v| = l$. Define 	\begin{center}
		$\beta = \left( {\begin{array}{ccccccccccccccccccccc}
				u   &  v    & z \\
				d_u &  d_v  & z
		\end{array} } \right)$
	\end{center}
	where $z\in X_n\setminus X_n\alpha$.	Then $\beta$ is well-defined and belongs to $T_n(l)$. Let $x\in X_n$. Thus $x\alpha\in X_n\alpha$ and $(d_{x\alpha})\alpha = x\alpha$. Accordingly,  $((x\alpha)\beta)\alpha =  (d_{x\alpha})\alpha = x\alpha$, that is,  $\alpha\beta\alpha = \alpha$. Therefore  $\alpha$ is regular.
\end{proof}

%%%%%%%%%%%%%%%%%%% Section 4 T* %%%%%%%%%%%%%%%%%%%%%%%%%%%%%%

\section{Regularity of $T^*_n(l)$} \label{Reg_T_star}

In this section, we prove that $T^*_n(l)$ is a regular semigroup. 

For $n=2$, we have $l=1$. It is easy to see that
$T^*_2(1) = \biggl\{
\begin{pmatrix}
	1 & 2\\
	1 & 2 
\end{pmatrix}, \begin{pmatrix}
	1 & 2\\
	2 & 1 
\end{pmatrix}
\biggl\}$ is  regular.  Now, we suppose that $n\geq 3$.

\begin{lemma}\label{RemK*1} Assume that $(n\geq 3$ is an odd integer and $l \geq \frac{n+1}{2})$ or $(n\geq 4$ is an even integer and $l > \frac{n}{2})$. Let $\alpha\in T^*_n(l)$. Then
	\begin{center}
		$
		\small{\alpha = \left( {\begin{array}{ccccccccccccccccccccc}
				1 & l+1 & 2 &  l+2 & \cdots & n-l & n & (n-l)+1 &  \cdots &  l  \\
				1\alpha & (l+1)\alpha & 2\alpha &  (l+2)\alpha & \cdots & (n-l)\alpha & n\alpha & ((n-l)+1)\alpha  & \cdots &  l\alpha
		\end{array} } \right)}
		$	
	\end{center}
	where $\{(n-l)+1, (n-l)+2,\dots,l\}\alpha\subseteq \{(n-l)+1, (n-l)+2,\dots,l\}$, and for each $1\leq x\leq n-l$, 
	\begin{center}
		$(x\alpha, (l+x)\alpha) \in \{(i, l+i)\colon 1\leq i\leq n-l\} \cup \{(l+i, i)\colon 1\leq i\leq n-l\}$
	\end{center}
	such that $\{x\alpha, (l+x)\alpha\}\cap \{y\alpha, (l+y)\alpha\} = \emptyset$ for all $x\neq y\in\{1,2,\dots, n-l\}$. 
\end{lemma}
\begin{proof} By Remark \ref{RemK}\,(1) and $\alpha$ preserves the length $l$, we can write
	\begin{center}
		$\small{\alpha = \left( {\begin{array}{ccccccccccccccccccccc}
				1 & l+1 & 2 &  l+2 & \cdots & n-l & n & (n-l)+1 &  \cdots &  l  \\
				1\alpha & (l+1)\alpha & 2\alpha &  (l+2)\alpha & \cdots & (n-l)\alpha & n\alpha & ((n-l)+1)\alpha  & \cdots &  l\alpha
		\end{array} } \right)}$
	\end{center}
	where for each $1\leq x\leq n-l$, 
	\begin{center}
		$(x\alpha, (l+x)\alpha) \in \{(i, l+i)\colon 1\leq i\leq n-l\} \cup \{(l+i, i)\colon 1\leq i\leq n-l\}$.
	\end{center}
	Suppose that there exist $x,y\in\{1,2,\dots, n-l\}$ such that $x\neq y$ but $\{x\alpha, (l+x)\alpha\}\cap \{y\alpha, (l+y)\alpha\} \neq \emptyset$. There are three possible cases to consider.
	
	\noindent\textbf{Case 1.}  $x\alpha = y\alpha$. Since $|(l+x) - x| = l$, it follows that $l = |(l+x)\alpha - x\alpha| = |(l+x)\alpha - y\alpha|$. Hence $|(l+x) - y| = l$. Then by Remark \ref{RemK}\,(1), we get $x = y$, a contradiction. 
	
	\noindent\textbf{Case 2.}  $x\alpha = (l+y)\alpha$. Since $|(l+y) - y| = l$, this implies that \[l = |(l+y)\alpha - y\alpha| = |x\alpha - y\alpha|.
	\]
	Thus $|x - y| = l$, which is impossible because  $|x-y|\leq (n-l)-1 < l -1 < l$.

	\noindent\textbf{Case 3.}  $(l+x)\alpha = (l+y)\alpha$. Since $|(l+y) - y| = l$, $l = |(l+y)\alpha - y\alpha| = |(l+x)\alpha - y\alpha|$. Hence $|(l+x) - y| = l$. Then by Remark \ref{RemK}\,(1), we get $x = y$, a contradiction. 
	
	\noindent Therefore $\{x\alpha, (l+x)\alpha\}\cap \{y\alpha, (l+y)\alpha\} = \emptyset$ for all $x\neq y\in\{1, \dots, n-l\}$ and so $\{1,\dots,n-l,l+1,l+2,\dots,n\}\alpha = \{1,2, \dots,n-l,l+1,l+2,\dots,n\}$. If there is an element $x\in \{(n-l)+1, (n-l)+2,\dots,l\}$ such that $x\alpha \in \{1,2,\dots, n-l,l+1, l+2,\dots,n\}$, then there exists $y\in \{1,2,\dots,n-l,l+1,l+2,\dots,n\}$ such that $|x\alpha - y\alpha| = l$. Thus $|x-y| = l$, contrary to Remark \ref{RemK}\,(1). Therefore  $\{(n-l)+1, (n-l)+2,\dots,l\}\alpha\subseteq \{(n-l)+1, (n-l)+2,\dots,l\}$. 
\end{proof}

\begin{theorem}\label{Regular*2} If $(n\geq 3$ is an odd integer and $l \geq \frac{n+1}{2})$ or $(n\geq 4$ is an even integer and $l > \frac{n}{2})$, then $T^*_n(l)$ is a regular semigroup.
\end{theorem}

\begin{proof} Assume that $(n\geq 3$ is an odd integer and $l \geq \frac{n+1}{2})$ or $(n\geq 4$ is an even integer and $l > \frac{n}{2})$. Let $\alpha$ be any element in $T^*_n(l)$. Then by Lemma \ref{RemK*1}, we can write 
	\begin{center}
		$\small{\alpha = \left( {\begin{array}{ccccccccccccccccccccc}
				1 & l+1 & 2 &  l+2 & \cdots & n-l & n & (n-l)+1 &  \cdots &  l  \\
				1\alpha & (l+1)\alpha & 2\alpha &  (l+2)\alpha & \cdots & (n-l)\alpha & n\alpha & ((n-l)+1)\alpha  & \cdots &  l\alpha
		\end{array} } \right)}
		$
	\end{center}
	where $\{(n-l)+1, (n-l)+2,\dots,l\}\alpha\subseteq \{(n-l)+1, (n-l)+2,\dots,l\}$, and for each $1\leq x\leq n-l$, 
	\begin{center}
		$(x\alpha, (l+x)\alpha) \in \{(i, l+i)\colon 1\leq i\leq n-l\} \cup \{(l+i, i)\colon 1\leq i\leq n-l\}$
	\end{center}
	such that $\{x\alpha, (l+x)\alpha\}\cap \{y\alpha, (l+y)\alpha\} = \emptyset$ for all $x\neq y\in\{1,\dots, n-l\}$. For $z\alpha\in\{(n-l)+1, (n-l)+2,\dots,l\}\alpha$, let  $d_{z\alpha}\in(z\alpha)\alpha^{-1}\subseteq \{(n-l)+1,  (n-l)+2,\dots,l\}$. Define
	\begin{center}
		$\beta = \left( {\begin{array}{ccccccccccccccccccccc}
				1\alpha & (l+1)\alpha & 2\alpha &  (l+2)\alpha & \cdots & (n-l)\alpha & n\alpha & z\alpha & w\\
				1 & l+1 & 2&  l+2 & \cdots & n-l & n & d_{z\alpha} & w
		\end{array} } \right)$
	\end{center}
	where $w\in \{(n-l)+1, (n-l)+2,\dots,l\}\setminus X_n\alpha$. It is easy to verify that $\beta\in T^*_n(l)$. Let $a\in X_n$. If $a\in \{1,2,\dots,n-l,l+1,l+2,\dots,n\}$, then $a(\alpha\beta\alpha) = ((a\alpha)\beta)\alpha = a\alpha$. If $a\in \{(n-l)+1, (n-l)+2,\dots,l\}$, then $a\alpha\in \{(n-l)+1, (n-l)+2,\dots,l\}\alpha$ and so $a(\alpha\beta\alpha) = ((a\alpha)\beta)\alpha = (d_{a\alpha})\alpha = a\alpha$. Thus $\alpha\beta\alpha = \alpha$. Therefore $\alpha$ is regular.
\end{proof}

From now on, we assume that $(n\geq 3$ is an odd integer and $l < \frac{n+1}{2})$ or $(n\geq 4$ is an even integer and $l\leq \frac{n}{2})$. To prove $T^*_n(l)$ is regular, we need the following seven lemmas.

\begin{lemma}\label{T^*l} Let $\alpha\in T^*_n(l)$. If $x,y\in X_n$ such that $x<y$ and $x\alpha = y\alpha$, then $y = x+2l$.
\end{lemma}
\begin{proof} Assume $x,y\in X_n$ such that $x<y$ and $x\alpha = y\alpha$. By  Remark \ref{RemK}\,(2), there exists $x+l\in X_n$ or $x-l\in X_n$. Since $|x-(x\pm l)| = l$, we obtain $|x\alpha-(x\pm l)\alpha| = l$. It follows that $(x\pm l)\alpha = x\alpha\pm l$. Using our assumption, we have
	\begin{center} $
		|(x\alpha\pm l)-x\alpha| = l \Rightarrow |(x\pm l)\alpha-y\alpha|  = l \Rightarrow |(x\pm l)-y|  = l \Rightarrow y= x \text{ or }  y = x\pm 2l $.
	\end{center}
	If $y = x$ or $y = x-2l$, then $y\leq x$, a contradiction. Thus $y = x+2l$.
\end{proof}

For $i\in\{1,2,\dots,l\}$, let $A_i = \{i,i+l,i+2l,\dots,i+m_il\}$ when $m_i$ is the maximum positive integer such that $i+m_il\leq n$. By  Remark \ref{RemK}\,(2), we can write
\begin{center}
	$X_n = \bigcup\limits_{i=1}^l A_i$		
\end{center}
where $A_i\cap  A_j=\emptyset$ for all $i\neq j\in\{1,2,\dots,l\}$. Moreover, if $x\in A_i$ and $y\in A_j$ for some $i\neq j\in\{1,2,\dots,l\}$, then $|x-y|\neq l$.  Note that $m_l \leq m_{l-1} \leq \cdots \leq m_2 \leq m_1$ and $|A_i| = m_i+1$ for all $1\leq i\leq l$.

\begin{lemma}\label{Cases22}  Let  $\alpha\in T_n^*(l)$ and $i\in\{1,2,\dots,l\}$. If $x\in A_i$ and $x\alpha\in A_j$ for some $j\in\{1,2,\dots,l\}$, then $A_i\alpha\subseteq A_j$.
\end{lemma}
\begin{proof} Let $x\in A_i$ be such that $x\alpha\in A_j$ for some $j\in\{1,2,\dots,l\}$. Suppose that $A_i\alpha\nsubseteq A_j$. Then there exits $y\in A_i$ such that $y\alpha\notin A_j$. If $y>x$, we let $u$ is the first element of $A_i$ such that $u>x$ and $u\alpha\notin A_j$. It follows that $u\leq y, u-l\in A_i$ and  $(u-l)\alpha\in A_j$. Since $|(u-l)\alpha -u\alpha| \neq l$, which implies that $|(u-l)-u| \neq l$, a contradiction. If $y<x$, we let $v$ is the last element of $A_i$ such that $v<x$ and $v\alpha\notin A_j$. It follows that $y\leq v, v+l\in A_i$ and  $(v+l)\alpha\in A_j$. Since $|v\alpha -(v+l)\alpha| \neq l$, this implies that $|v-(v+l)| \neq l$, a contradiction. 	Therefore  $A_i\alpha\subseteq A_j$.
\end{proof}

\begin{lemma}\label{Cases2}  Let  $\alpha\in T_n^*(l)$ and $i,j\in\{1,2,\dots,l\}$ such that $x\in A_i$ and $y\in A_j$. If $i\neq j$, then $x\alpha\neq y\alpha$.
\end{lemma}
\begin{proof} Assume that  $i\neq j$. Thus  $A_i\cap A_j = \emptyset$. This implies that $y\notin A_i$ and $x\neq y$.  We assume that $x<y$ and $x\alpha = y\alpha$. Let $x = i+kl$ for some $k\in\{0,1,\dots,m_i\}$.   By Lemma \ref{T^*l}, we get 
	\begin{center}
		$i+(k+2)l = (i+kl)+2l = x+2l = y \leq n$.
	\end{center}	
	Hence $y = i+(k+2)l\in A_i$, a contradiction. Therefore $x\alpha\neq y\alpha$. 
\end{proof}

The next remark follows directly from  Lemma \ref{Cases2}.

\begin{remark}\label{RRRem5} 
	Let $\alpha\in T^*_n(l)$ and $x\in X_n$. If $x\in A_i\alpha$ for some $i\in\{1,2,\dots,l\}$, then $x\notin A_j\alpha$ for all $j\in \{1,2,\dots,l\}, j\neq i$.
\end{remark}

\begin{lemma}\label{case2}  Let $i\in\{1,2,\dots,l\}$ and $\alpha\in T_n^*(l)$. If $m_i\geq 3$, then either $i\alpha< (i+l)\alpha< (i+2l)\alpha< \cdots < (i+m_il)\alpha$ or $i\alpha > (i+l)\alpha > (i+2l)\alpha > \cdots > (i+m_il)\alpha$. 
\end{lemma}
\begin{proof} Suppose that $m_i\geq 3$. That is, $i,i+l,i+2l,i+3l\in X_n$. We see that \break $(i\alpha, (i+l)\alpha, (i+2l)\alpha, (i+3l)\alpha)\in \{(i\alpha, i\alpha+l, i\alpha+2l, i\alpha+3l), (i\alpha, i\alpha+l,   i\alpha+2l, \break i\alpha+l), (i\alpha, i\alpha+l, i\alpha, i\alpha+l), (i\alpha, i\alpha+l, i\alpha,  i\alpha-l), (i\alpha, i\alpha-l, i\alpha,  i\alpha+l),(i\alpha, i\alpha-l, i\alpha, \break i\alpha-l), (i\alpha, i\alpha-l, i\alpha-2l, i\alpha-l), (i\alpha, i\alpha-l, i\alpha-2l, i\alpha-3l)\}$.  If $(i+3l)\alpha\notin  \{i\alpha+3l, i\alpha-3l\}$, then $(i+3l)\alpha\in\{i\alpha+l, i\alpha-l\}$ and so  $|i\alpha - (i+3l)\alpha| = |i\alpha - (i\alpha\pm l)| = l$ but $|i-(i+3l)| = 3l\neq l$, a contradiction. Hence $(i+3l)\alpha\in \{i\alpha+3l, i\alpha-3l\}$. It follows that $(i\alpha, (i+l)\alpha, (i+2l)\alpha,  (i+3l)\alpha) \in \{ (i\alpha, i\alpha+l, i\alpha+2l, i\alpha+3l), (i\alpha, i\alpha-l,    i\alpha-2l, i\alpha-3l)\}$.  	For the case $(i\alpha, (i+l)\alpha, (i+2l)\alpha, (i+3l)\alpha) = (i\alpha, i\alpha+l, i\alpha+2l, i\alpha+3l)$. It is clear that either $(i+4l)\alpha = i\alpha+2l$ or $(i+4l)\alpha = i\alpha+4l$. If $(i+4l)\alpha = i\alpha+2l$, then $|(i+l)\alpha - (i+4l)\alpha| = |(i\alpha+l) - (i\alpha+2l)| = l$ but $|(i+l) - (i+4l)| = 3l\neq l$, a contradiction. Thus $(i+4l)\alpha = i\alpha+4l$. We finally obtain that $(i+kl)\alpha = i\alpha+kl$ for all $0\leq k\leq m_i$. Therefore
	\begin{center}
		$i\alpha< (i+l)\alpha< (i+2l)\alpha< \cdots < (i+m_il)\alpha$.
	\end{center}
	For the case $(i\alpha, (i+l)\alpha, (i+2l)\alpha, (i+3l)\alpha) =  (i\alpha, i\alpha-l, i\alpha-2l, i\alpha-3l)\}$. It is obvious that either $(i+4l)\alpha = i\alpha-2l$ or $(i+4l)\alpha = i\alpha-4l$. If $(i+4l)\alpha = i\alpha-2l$, then $|(i+l)\alpha - (i+4l)\alpha| = |(i\alpha-l) - (i\alpha-2l)| = l$ while $|(i+l) - (i+4l)| = 3l\neq l$, a contradiction. Thus $(i+4l)\alpha = i\alpha-4l$. Similarly, we obtain  $(i+kl)\alpha = i\alpha-kl$ for all $0\leq k\leq m_i$. Therefore
	\begin{center}
		$i\alpha > (i+l)\alpha > (i+2l)\alpha > \cdots > (i+m_il)\alpha$.
	\end{center}
\end{proof}

Lemma \ref{case2} yields the following remark directly.

\begin{remark}\label{RRem2} 
	$|A_i\alpha| = |A_i|$ for all $\alpha\in T^*_n(l)$ and all $m_i\geq 3$.
\end{remark}

\begin{lemma}\label{Cases222}  Let  $\alpha\in T_n^*(l)$ and $i,j\in\{1,2,\dots,l\}$ such that $m_j\geq 3$. If $m_i < m_j$, then $x\notin  A_j\alpha$ for all $x\in A_i$.
\end{lemma}
\begin{proof} Let $m_i < m_j$. Assume that $x\in  A_j\alpha$ for some $x\in A_i$. Then by Lemma \ref{Cases22}, we get $A_j\alpha\subseteq A_i$. It follows that $|A_j\alpha| \leq |A_i| = m_i + 1 < m_j+1 = |A_j|$. That is,  $|A_j\alpha| < |A_j|$, which contradicts Remark \ref{RRem2}. Hence $x\notin  A_j\alpha$ for all $x\in A_i$.
\end{proof}

\begin{lemma}\label{CCCases}  Let $\alpha\in T_n^*(l)$ and $i\in\{1,2,\dots,l\}$ such that $m_i\geq 3$. Let $Z_{m_i} = \{j\in \{1,2,\dots,l\} \colon m_j=m_i\}$. Then $A_j\alpha = A_{j\sigma}$ for all $j\in Z_{m_i}$ and for some $\sigma\in G(Z_{m_i})$, the permutation group on $Z_{m_i}$.
\end{lemma}
\begin{proof} Let $d$ be an integer such that $d\geq 0$ and $m_1-d\geq 3$. Consider the proposition
	\begin{center}
		$P(d)\colon  A_j\alpha = A_{j\sigma} \text{ for all } j\in Z_{m_1-d} \text{ and for some } \sigma\in G(Z_{m_1-d})$.
	\end{center}	
	We shall prove this by a strong induction on $d$.

	First, we prove that $P(0)$ holds. That is,  $A_j\alpha = A_{j\sigma}$ for all $j\in Z_{m_1}$ and for some $\sigma\in G(Z_{m_1})$.  Note that $m_1\geq m_k$ for all $1\leq k\leq l$. Suppose that $Z_{m_1} =\{j_1,j_2,\dots,j_r\}$ such that $j_1\neq j_2\neq \cdots \neq j_r$. It is obvious that $1\in Z_{m_1}$. For convenience, we assume that $j_1 = 1$. That is, $j_2,j_3,\dots,j_r\in \{2,3,\dots,l\}$ where $m_{j_2} = m_{j_2} = \cdots = m_{j_r} = m_1\geq 3$. 	It follows from Lemma \ref{case2} and Lemma \ref{Cases222} that $A_{j_1}\alpha = A_1\alpha = A_{k_1}$ for some $k_1\in Z_{m_1}$. Similarly, we have $A_{j_2}\alpha = A_{k_2}$ for some $k_2\in Z_{m_1}$. If $k_2 = k_1$, we obtain $A_{1}\alpha = A_{j_2}\alpha$, that is, there exist $x\in A_1, y\in A_{j_2}$ such that $x\alpha = y\alpha$. Then by Lemma \ref{Cases2}, we get $j_2 = 1 = j_1$, a contradiction. Hence $k_2\neq k_1$. We also can show that $A_{j_3}\alpha = A_{k_3}$ for some $k_3\in Z_{m_1}\setminus\{k_1,k_2\}$.  At the end of this process we obtain $A_{j_r}\alpha = A_{k_r}$ for some $k_r\in Z_{m_1}\setminus\{k_1,k_2,\dots,k_{r-1}\}$. Since $\{k_1,k_2,\dots,k_r\}\subseteq Z_{m_1}$ and $|Z_{m_1}| = |\{k_1,k_2,\dots,k_r\}| = r$, which implies that $Z_{m_1} = \{k_1,k_2,\dots,k_r\}$. We define $j_t\sigma = k_t$ for all $1\leq t\leq r$. It is clear that $\sigma\colon Z_{m_1} \to Z_{m_1}$ is a bijection. Thus $\sigma\in G(Z_{m_1})$ and $A_{j_t}\alpha = A_{{j_t}\sigma}$ for all $1\leq t \leq r$. Therefore $P(0)$ holds.
	
	Now, assume that $P(d)$ holds for all $0\leq d\leq k-1$ where $m_1-(k-1)\geq 3$. Then $m_1 -d \geq m_1-(k-1)\geq 3$ for all $0\leq d\leq k-1$. We prove $P(k)$ holds  where $m_1-k\geq 3$. Suppose that $Z_{m_1-k} =\{q_1,q_2,\dots,q_s\}$ such that $q_1\neq q_2\neq \cdots \neq q_s$. 	Since $k>0$, which implies that $1\notin  Z_{m_1-k}$ and so  $q_1,q_2,\dots,q_s\in \{2,3,\dots,l\}$ where $m_{q_1} = m_{q_2} = \cdots = m_{q_s} = m_1-k\geq 3$. We show that if $x\in A_p$ where $m_1-k < m_p$, then $x\notin A_{q_t}\alpha$ for all $1\leq t\leq s$. Let $x\in A_p$ such that $m_1-k < m_p$. Then $m_p = m_1 - d$ for some $d\in\{0,1,\dots, k-1\}$, that is, $p \in Z_{m_1-d}$. If $x\in A_{q_t}\alpha$ for some $q_t\in Z_{m_1-k}$, then $x=y\alpha$ for some $y\in A_{q_t}$. By our induction hypothesis, we have $A_j\alpha = A_{j\sigma} \text{ for all } j\in Z_{m_1-d} \text{ where } \sigma\in G(Z_{m_1-d})$. Since $p\in Z_{m_1-d}$, there exists unique $u \in Z_{m_1-d}$ such that $u\sigma = p$. Thus we get $A_u\alpha = A_{u\sigma}  = A_p$ and so $x = z\alpha$ for some $z\in A_{u}$. Since $q_t\neq u$, we obtain $y\alpha \neq z\alpha$ by Lemma \ref{Cases2}, this is a contradiction. Therefore  $x\notin A_{q_t}\alpha$ for all $1\leq t\leq s$. It follows from Lemma \ref{case2} and Lemma \ref{Cases222} that for each $q_t\in Z_{m_1-k}$, $A_{q_t}\alpha = A_{q'_t}$ for some $q'_t\in Z_{m_1-k}$. If $q'_t = q'_j$ where $q_t, q_j\in Z_{m_1-k}$, we obtain $A_{q_t}\alpha = A_{q'_t} = A_{q'_j} = A_{q_j}\alpha$. Thus there exist $v\in A_{q_t}, w\in A_{q_j}$ such that $v\alpha = w\alpha$, we get $q_t = q_j$ by Lemma \ref{Cases2}. This implies that $Z_{m_1-k} =\{q'_1,q'_2,\dots,q'_s\}$. We define $\delta\colon Z_{m_1-k} \to Z_{m_1-k}$ by $q_t\delta = q'_t$ for all $1\leq t \leq s$. It is clear that  $\delta\in G(Z_{m_1-k})$ and $A_{q_t}\alpha = A_{q'_t} = A_{q_t\delta}$ for all $1\leq t \leq s$. So $P(k)$ holds.
\end{proof}

\begin{remark}\label{Rmark33}  Let  $\alpha\in T_n^*(l)$ and $i,j\in\{1,2,\dots,l\}$ such that $m_j\geq 3$. If $m_i < m_j$, then $x\notin  A_i\alpha$ for all $x\in A_j$. Indeed, if there is $x\in  A_i\alpha$ for some $x\in A_j$ where $m_i < m_j$, then $x = y\alpha$ for some $y\in A_i$. Since $m_j\geq 3$, $A_j = A_k\alpha$ for some $m_k = m_j\geq 3$ by Lemma \ref{CCCases}. Thus $x = z\alpha$ for some $z\in A_k$. Now, we have $y\alpha = z\alpha$ where $y\in A_i$ and $z\in A_k$ such that $m_i < m_k$. It follows from Lemma \ref{Cases2} that $i = k$, which is a contradiction.
\end{remark}

\begin{lemma}\label{Cases*2}  Let $\alpha\in T_n^*(l)$ and $X_n\alpha \neq X_n$. If $u\in  X_n\setminus X_n\alpha$, then  $(u = j\in\{1,2,\dots,l\}$ such that $m_j = 2$ and $j+2l\in X_n\alpha)$ or $(u = j+2l$ where $j\in\{1,2,\dots,l\}$ such that $m_j = 2$ and $j\in X_n\alpha)$. In particular, $j+l\in X_n\alpha$ for all $j\in\{1,2,\dots,l\}$ such that $m_j =2$.
\end{lemma}
\begin{proof} Assume that $u\in X_n\setminus X_n\alpha$. For $i\in\{1,2,\dots,l\}$ with $m_i\geq 3$, by Lemma \ref{CCCases} we have $A_i = A_j\alpha\subseteq X_n\alpha$ for some $j\in\{1,2,\dots,l\}$ such that $m_j = m_i$. Hence $u\notin A_i$ for all $i\in\{1,2,\dots,l\}$ such that $m_i\geq 3$. That is, $u\in A_i$ for some $i\in\{1,2,\dots,l\}$ such that $m_i\in\{1,2\}$. Let  $i_1,i_2,\dots,i_r\in\{1,2,\dots,l\}$ such that $m_{i_1}= m_{i_2}= \cdots = m_{i_r} =1$ and $j_1,j_2,\dots,j_s\in\{1,2,\dots,l\}$ such that $m_{j_1} = m_{j_2}= \cdots = m_{j_s} = 2$. 	Note that $i_t+ql\notin X_n$ for all $1\leq t\leq r$ and all $q\geq 2$. There are two possible cases to consider.
	
	\noindent\textbf{Case 1.} $\{i_1, i_2\dots,i_r\} = \emptyset$ or [$\{i_1,i_2,\dots,i_r\} \neq \emptyset$ and $(i_t\alpha, (i_t+l)\alpha)\in  \{(i_k, i_k+l), \break (i_k+l, i_k)\colon k\in\{1,2,\dots,r\}\}$ for all $1\leq t\leq r$]. Since $X_n\alpha\neq X_n$, by Lemma \ref{Cases22} and Remark \ref{RRRem5} there exists $j_p\in\{j_1,j_2,\dots,j_s\}$ such that 
	$(j_p\alpha, (j_p+l)\alpha, (j_p+2l)\alpha) \in\{(j_q, j_q+l, j_q), (j_q+l, j_q, j_q+l), (j_q+l, j_q+2l,   j_q+l), (j_q+2l, j_q+l, j_q+2l)\colon q\in\{1,2,\dots,s\}\}$. 
	
	If 	$(j_p\alpha, (j_p+l)\alpha, (j_p+2l)\alpha) \in\{(j_q, j_q+l, j_q), (j_q+l, j_q, j_q+l)\colon q\in\{1,2,\dots,s\}\}$, it is easy to see that $j_q+2l\notin X_n\alpha$ where $j_q+l\in X_n\alpha$ and $m_{j_q} = 2$. In this case, we have $u = j_q+2l$ such that $m_{j_q} = 2$ and $j_q\in X_n\alpha$.
	
	If 	$(j_p\alpha, (j_p+l)\alpha, (j_p+2l)\alpha) \in\{(j_q+l, j_q+2l, j_q+l), (j_q+2l, j_q+l, j_q+2l)\colon q\in\{1,2,\dots,s\}\}$, then we get $j_q\notin X_n\alpha$ and $j_q+l\in X_n\alpha$ such that $m_{j_q} = 2$. In this case, we have $u = j_q$ such that $m_{j_q} = 2$ and $j_q+2l\in X_n\alpha$.

	\noindent\textbf{Case 2.} 	$\{i_1,i_2,\dots,i_r\} \neq \emptyset$  and $(i_t\alpha, (i_t+l)\alpha)\in\{(j_k, j_k+l), (j_k+l, j_k),  (j_k+l, j_k+2l), (j_k+2l, j_k+l)\colon k\in\{1,2,\dots,s\}\}$ for some $t\in\{1,2,\dots,r\}$. Clearly, $j_k+l\in X_n\alpha$ when $k\in\{1,2,\dots,s\}$. Now, we partition into four subcases as follows:
	
	\textbf{Subcase 2.1} $(i_t\alpha, (i_t+l)\alpha) = (j_k, j_k+l)$ for some $k\in\{1,2,\dots,s\}$. Suppose that $j_k+2l\in X_n\alpha$. That is, there exists $x\in X_n$ such that $x\alpha = j_k+2l$. Since $|(i_t+l)\alpha - x\alpha| = |(j_k+l)-(j_k+2l)| = l$, we obtain $|(i_t+l) - x| = l$. It follows that $x = i_t$ or $x = i_t+2l$. If $x = i_t$, then $j_k+2l = x\alpha =i_t\alpha = j_k$. Hence $2l=0$, this is impossible. If $x = i_t+2l$, then $x\notin X_n$, a contradiction. Therefore $j_k+2l\notin X_n\alpha$.
	
	\textbf{Subcase 2.2} $(i_t\alpha, (i_t+l)\alpha) = (j_k+l, j_k)$ for some $k\in\{1,2,\dots,s\}$. Suppose that $j_k+2l\in X_n\alpha$. That is, there exists $x\in X_n$ such that $x\alpha = j_k+2l$. Since $|i_t\alpha - x\alpha| = |(j_k+l)-(j_k+2l)| = l$, we obtain $|i_t - x| = l$. Thus $x=i_t\pm l$. If $x = i_t+l$, then $j_k+2l = x\alpha = (i_t+l)\alpha = j_k$. Hence $2l=0$, this is impossible. If $x = i_t-l$, then $x\notin X_n$ (since $1\leq i_t\leq l$), a contradiction. Therefore $j_k+2l\notin X_n\alpha$.
	
	\textbf{Subcase 2.3} $(i_t\alpha, (i_t+l)\alpha) = (j_k+l, j_k+2l)$ for some $k\in\{1,2,\dots,s\}$. Suppose that $j_k\in X_n\alpha$. That is, there exists $x\in X_n$ such that $x\alpha = j_k$. Since $|i_t\alpha - x\alpha| = |(j_k+l)-j_k| = l$, we obtain $|i_t - x| = l$. Thus $x=i_t\pm l$. Since $i_t-l\notin X_n$, we have $x = i_t+l$. It follows that $j_k = x\alpha = (i_t+l)\alpha = j_k+2l$. Hence $2l=0$, this is impossible. Therefore $j_k\notin X_n\alpha$.
	
	\textbf{Subcase 2.4} $(i_t\alpha, (i_t+l)\alpha) = (j_k+2l, j_k+l)$ for some $k\in\{1,2,\dots,s\}$. Suppose that $j_k\in X_n\alpha$. That is, there exists $x\in X_n$ such that $x\alpha = j_k$. Since $|(i_t+l)\alpha - x\alpha| = |(j_k+l)-j_k| = l$, we obtain $|(i_t+l) - x| = l$. This implies that that $x = i_t$ or $x = i_t+2l$. Since $i_t+2l\notin X_n$, we have $x = i_t$. It follows that $j_k = x\alpha =i_t\alpha = j_k+2l$ and so $2l=0$, this is impossible. Hence $j_k\notin X_n\alpha$.
	
	From Subcases 2.1 and 2.2, we have $u = j_k+2l$ such that $m_{j_k} = 2$ and $j_k\in X_n\alpha$. And from Subcases 2.3 and 2.4, we have $u = j_k$ such that $m_{j_k} = 2$ and $j_k+2l\in X_n\alpha$.
\end{proof}

\begin{theorem}\label{RRegular} If $(n\geq 3$ is an odd integer and $l < \frac{n+1}{2})$ or $(n\geq 4$ is an even integer and $l\leq \frac{n}{2})$, then $T^*_n(l)$ is a regular semigroup.
\end{theorem}
\begin{proof} Suppose $(n\geq 3$ is odd integer and $l < \frac{n+1}{2})$ or $(n\geq 4$ is an even integer and $l\leq \frac{n}{2})$. Let $\alpha\in T^*_n(l)$. 
	%	We consider the cases $X_n\alpha = X_n$ and $X_n\alpha \neq X_n$ separately.
	
	\noindent\textbf{Case 1.} $X_n\alpha = X_n$. Then $\alpha$ is bijective. Since $T_n$ is a regular semigroup and $\alpha\in T^*_n(l)\subseteq T_n$, there exists $\beta\in T_n$ such that $\alpha = \alpha\beta\alpha$. We prove that $\beta\in T_n^*(l)$. Let $x,y\in X_n$. Then there are $u,v\in X_n$ such that $u\alpha = x$ and $v\alpha = y$. Consider
	\begin{align*}
	|x-y| = l 	&  \: \Leftrightarrow \:|u\alpha - v\alpha| = l \\
				&  \: \Leftrightarrow \: |u(\alpha\beta\alpha) - v(\alpha\beta\alpha)| = l\\
				&  \: \Leftrightarrow \: |(u\alpha\beta)\alpha - (v\alpha\beta)\alpha| = l \\
				&  \: \Leftrightarrow \:|(u\alpha)\beta - (v\alpha)\beta| = l \\
				&  \: \Leftrightarrow \: |x\beta - y\beta| = l.
	\end{align*}
		Hence $\beta\in T_n^*(l)$.

	\noindent\textbf{Case 2.} $X_n\alpha \neq X_n$.  For $x\in X_n\alpha$, we choose and fix $d_x\in x\alpha^{-1}$. Thus $d_x\alpha = x$ for all $x\in X_n\alpha$. 	For $x\notin X_n\alpha$, by Lemma \ref{Cases*2}, we have $(x = j\in\{1,2,\dots,l\}$ such that $m_j = 2$ and $j+l, j+2l\in X_n\alpha)$ or $(x = k+2l$ where $k\in\{1,2,\dots,l\}$ such that $m_k = 2$ and $k, k+l\in X_n\alpha)$.	Define $\beta\in T_n$ by
	\begin{center}
		$\beta = \begin{pmatrix}
			v   & \cdots & j   &  j+l     & j+2l  & \cdots & k & k+l & k+2l \\
			d_v & \cdots & d_{j+2l} &  d_{j+l} & d_{j+2l} & \cdots & d_k & d_{k+l} & d_k
		\end{pmatrix}
		$
	\end{center}
	where  $v\in X_n\alpha\setminus [\{j,j+l,j+2l\colon j\notin X_n\alpha  \text{ and } m_j=2\}\cup \{k,k+l,k+2l\colon k+2l\notin X_n\alpha  \text{ and } m_k=2\}]$. Let $z\in X_n$. Then $z\alpha\in X_n\alpha$ and $z(\alpha\beta\alpha) = ((z\alpha)\beta)\alpha = (d_{z\alpha})\alpha = z\alpha$. Hence $\alpha\beta\alpha= \alpha$. We prove that $\beta\in T^*_n(l)$. Let $x,y\in X_n$. If  $x,y\in X_n\alpha$, by the same proof as given for \textbf{Case 1} then is 
	\begin{center}
		$|x-y| = l \;\Leftrightarrow \; |x\beta - y\beta| = l$.
	\end{center}
	Assume that $x\notin X_n\alpha$ (or $y\notin X_n\alpha$).  Let $x = j\in\{1,2,\dots,l\}$ such that $m_j = 2$ and $j+l, j+2l\in X_n\alpha$. From the proof of Cases 1, 2 we have $(d_{j+l}, d_{j+2l})\in\{(a,a+l),  \break (a+l,a), (a+l, a+2l), (a+2l, a+l)\}\cup\{(b,b+l),  (b+l, b)\}$ for some $a,b\in\{1,2,\dots,l\}$ such that $m_a = 2$ and $m_b = 1$. Thus  $|d_{j+2l} - d_{j+l}| = l$ where $d_{j+l}, d_{j+2l}\in A_a$ or  $d_{j+l}, d_{j+2l}\in A_b$ such that $m_a = 2$ and $m_b = 1$. Then we obtain 
		\begin{align*}
		|x-y| = l 	&  \: \Rightarrow \:|j-y| = l  \\
		&  \: \Rightarrow\: y = j+l\\
				&  \: \Rightarrow \: |x\beta-y\beta| = |j\beta-(j+l)\beta|  = |d_{j+2l} - d_{j+l}| = l.
	\end{align*}
	Conversely, 
	\begin{align*}
	|x\beta-y\beta| = l 	&  \: \Rightarrow \: |j\beta-y\beta| = l  \\
	&  \: \Rightarrow\: |d_{j+2l}-y\beta| = l\\
	&  \: \Rightarrow \: y\beta = d_{j+l} \\
		&  \: \Rightarrow \: y = j+l\\
		 &  \: \Rightarrow \: |x-y| = l.	
	\end{align*}
	Let $x = k+2l$ for some $k\in\{1,2,\dots,l\}$ such that $m_k = 2$ and $k, k+l\in X_n\alpha$. From the proof of Cases 1, 2 we have $(d_{k}, d_{k+l})\in\{(a,a+l), (a+l,a), (a+l,  a+2l), \break (a+2l,  a+l)\}\cup\{(b,b+l),  (b+l, b)\}$ for some $a,b\in\{1,2,\dots,l\}$ such that $m_a = 2$ and $m_b = 1$. Hence  $|d_{k} - d_{k+l}| = l$ where $d_{k}, d_{k+l}\in A_a$ or  $d_{k}, d_{k+l}\in A_b$ such that $m_a = 2$ and $m_b = 1$. Then we have
	\begin{center}
		$|x-y| = l \;\Rightarrow \;|(k+2l)-y| = l \;\Rightarrow \;y = k+l \Rightarrow \;|x\beta-y\beta| = |d_{k} - d_{k+l}| = l$. 
	\end{center}
	On the other hand,  
	\begin{align*}
		|x\beta-y\beta| = l	&  \: \Rightarrow \: |(k+2l)\beta-y\beta| = l   \\
		&  \: \Rightarrow\: |d_{k}-y\beta| = l \\
		&  \: \Rightarrow \: y\beta = d_{k+l}\\
		&  \: \Rightarrow \: y = k+l \\
		&  \: \Rightarrow \: |x-y| = l.
	\end{align*}
	Therefore $\beta\in T^*_n(l)$.
\end{proof}

\vskip0.2cm
We end this section by showing that $T^*_n(l)$ is not the largest regular subsemigroup of $T_n(l)$. 

(a) Let $n = 5$ and $l = 3$. That is, $l \geq \frac{5+1}{2}$. We define $\alpha = \begin{pmatrix}
	1 & 2  & 3 & 4 & 5\\
	1 & 1  & 3 & 4 & 4
\end{pmatrix}$. It is obvious that $\alpha\in T_5(3)$ and $\alpha = \alpha \alpha \alpha $. Hence $\alpha$ is a regular element in $T_5(3)$ but $\alpha\notin T^*_5(3)$.

(b) Let $n = 6$ and $l = 2$. Then $l \leq \frac{6}{2}$. We define $\alpha = \begin{pmatrix}
	1 & 2  & 3 & 4 & 5 & 6\\
	1 & 1  & 3 & 3 & 5 & 5
\end{pmatrix}$. It is clear that $\alpha\in T_6(2)$ and $\alpha = \alpha \alpha \alpha $. Thus $\alpha$ is a regular element in $T_6(2)$ but $\alpha\notin T^*_6(2)$.

\section*{Acknowledgment}
We would like to thank the referees for
their comments and suggestions on the manuscript. 

\section*{Declarations}
The author declares that there is no conflicts of interest.

\end{document}